\documentclass[review]{elsarticle}

\usepackage{bm,amssymb,amsmath,mathrsfs}
\usepackage{amsthm}
\usepackage{lineno}
\usepackage[usenames]{color}

\usepackage{dcolumn}

\newcommand{\bq}{\begin{equation}}
\newcommand{\eq}{\end{equation}}
\newcommand{\bc}{\begin{center}}
\newcommand{\ec}{\end{center}}
\newcommand{\bit}{\begin{itemize}}
\newcommand{\eit}{\end{itemize}}
\newcommand{\ben}{\begin{enumerate}}
\newcommand{\een}{\end{enumerate}}

\theoremstyle{plain}
\newtheorem{theorem}{Theorem}[section]
\newtheorem*{theorem*}{Theorem}
\newtheorem{proposition}[theorem]{Proposition}

\newtheorem{remark}[theorem]{Remark}

\begin{document}

\journal{(internal report CC23-9)}

\begin{frontmatter}

\title{Scaling behavior for the median of the Poisson distribution of order $k$}

\author[cc]{S.~R.~Mane}
\ead{srmane001@gmail.com}
\address[cc]{Convergent Computing Inc., P.~O.~Box 561, Shoreham, NY 11786, USA}

\begin{abstract}
This note analyzes properties of the median $\nu$ of the Poisson distribution of order $k$.
Given a value for the median in the interval $\nu\in[1,k]$, an equation to calculate the corresponding value of the rate parameter $\lambda$ is derived.
Numerical evidence is presented that the value of the median exhibits many scaling properties,
which permit one to formulate parameterizations of the value of the median in various domains of the parameter space $(k,\lambda)$.
In all cases, the relevant quantities to calculate are $\nu/k$ and $\mu/k$, where $\mu$ is the mean.
\end{abstract}

\vskip 0.25in


\end{frontmatter}

\newpage
\setcounter{equation}{0}
\section{\label{sec:intro} Introduction}
The Poisson distribution of order $k$ is a special case of a compound Poisson distribution introduced by Adelson \cite{Adelson1966}.
For a (possibly infinite) tuple $\bm{a}=(a_1,a_2,\dots)$ and $x\in\mathbb{R}$, the probability generating function (pgf)
of a compound Poisson distribution is (eq.~(1) in \cite{Adelson1966})
\bq
\label{eq:Adelson_pgf}
f(\bm{a},x) = \exp\Bigl(-\sum_i a_i\Bigr)\exp\Bigl(\sum_i a_ix^i\Bigr) \,.
\eq
The Poisson distribution of order $k$ is the special case where $a_1=\dots=a_k=\lambda$, where $\lambda>0$, and all the other $a_i$ are zero.
For $k=1$ it is the standard Poisson distribution.
Although exact expressions for the mean and variance of the Poisson distribution of order $k$ are known \cite{PhilippouMeanVar},
exact results for its median and mode are difficult to obtain.
We denote the mean and variance by $\mu_k(\lambda)$ and $\sigma^2_k(\lambda)$, respectively.
Their values are as follows (from \cite{PhilippouMeanVar}, with correction of a misprint)
\begin{subequations}
\begin{align}
\label{eq:mean_Poisson_orderk}
\mu_k(\lambda) &= (1+\dots+k)\lambda = \frac12\,k(k+1)\lambda \,.
\\
\label{eq:variance_Poisson_orderk}
\sigma_k^2(\lambda) &= (1^2+\dots+k^2)\lambda = \frac16\,k(k+1)(2k+1)\lambda \,.
\end{align}
\end{subequations}
We denote the median and mode by $\nu_k(\lambda)$ and $m_k(\lambda)$, respectively.
Some exact results and also upper/lower bounds for the mode have been published in \cite{PhilippouFibQ,GeorghiouPhilippouSaghafi,KwonPhilippou}.
It was shown in \cite{Mane_Poisson_k_CC23_3} that for fixed $k\ge1$, the median is zero if and only if $\lambda \le (\ln2)/k$ (this result is well-known for $k=1$).
In a recent note \cite{Mane_Poisson_k_CC23_3},
the author presented conjectured expressions for the median and the mode of the Poisson distribution of order $k$, for fixed $k\ge2$ and sufficiently large values of $\lambda$.
First define the parameter $\kappa=k(k+1)/2$.
Fix the value of the mean to be an integer $\mu_k(\lambda)=\kappa\lambda=n$.
Numerical studies reported in \cite{Mane_Poisson_k_CC23_3} indicate that the value of the median and mode are respectively
\begin{subequations}
\begin{align}
\label{eq:Poisson_k_median_int_n}
\nu_k(n/\kappa) &= n - \biggl\lfloor\frac{k+4}{8}\biggr\rfloor \qquad (n \ge \kappa) \,,
\\
\label{eq:Poisson_k_mode_int_n}
m_k(n/\kappa) &= n - \biggl\lfloor\frac{3k+5}{8}\biggr\rfloor \qquad (n \ge 2\kappa) \,.
\end{align}
\end{subequations}
In this note, we shall study the behavior of the median for $n<\kappa$, i.e.~$0<\lambda<1$.
Given a value for the median in the interval $\nu\in[1,k]$,
an equation to calculate the corresponding value of the rate parameter $\lambda$ is derived.
We shall present numerical evidence that the value of the median exhibits many scaling properties,
which permit one to formulate parameterizations of the value of the median in various domains of the parameter space $(k,\lambda)$.
Particular attention is paid to the interval $\nu\in[1,k]$.

The structure of this paper is as follows.
Sec.~\ref{sec:notation} presents basic definitions and notation employed in this note.
Sec.~\ref{sec:median1k} derives equations and (approximate) solutions for the rate parameter $\lambda$ when the median lies in the interval $\nu\in[1,k]$.
Sec.~\ref{sec:median_eq_k} analyzes the case where the median equals $k$.
Secs.~\ref{sec:median_scaling1} and \ref{sec:median_scaling2} present graphical evidence for formulas to
parameterize the value of the median in various domains of the parameter space $(k,\lambda)$.
Sec.~\ref{sec:conc} concludes.

\newpage
\setcounter{equation}{0}
\section{\label{sec:notation}Basic notation and definitions}
We have already introduced the parameter $\kappa=k(k+1)/2$.
We denote the mean by $\mu$, the median by $\nu$ and the mode by $m$ (with pertinent subscripts, etc.~to denote the dependence on $k$ and $\lambda$, see below).
For the median, we follow the exposition in \cite{AdellJodraPoisson2005}:
if $Y_{k,\lambda}$ is a random variable which is Poisson distributed with order $k$ and parameter $\lambda$,
the median is defined as the smallest integer $\nu$ such that $P(Y_{k,\lambda} \le \nu) \ge \frac12$.
With this definition, the median is unique and is always an integer.
For fixed $k$, the value of the median increases in unit steps as the value of $\lambda$ increases.
The mode is defined as the location(s) of the {\em global maximum} of the probability mass function.
It is known that the mode may not be unique.
For the standard Poisson distribution with parameter $\lambda$, the mode equals $\lfloor\lambda\rfloor$ if $\lambda\not\in\mathbb{N}$,
but both $\lambda-1$ and $\lambda$ are modes if $\lambda\in\mathbb{N}$.
Unlike the median, for fixed $k$, the value of the mode can increase by more than unity as the value of $\lambda$ increases.
See results in \cite{PhilippouFibQ,KwonPhilippou,Mane_Poisson_k_CC23_5}.

We adopt the notation by Kostadinova and Minkova \cite{KostadinovaMinkova2013}
and denote the probability mass function by $p_n = P(Y_{k,\lambda}=n)$.
The above authors published the following expression for the pmf, in terms of combinatorial sums
(Theorem 1 in \cite{KostadinovaMinkova2013}, with slight changes of notation)
\bq
\label{eq:KM_Thm1}
\begin{split}
  p_0 &= e^{-k\lambda} \,,
  \\
  p_n &= e^{-k\lambda}\,\sum_{j=1}^n \binom{n-1}{j-1}\,\frac{\lambda^j}{j!} \qquad (n=1,2,\dots,k) \,,
  \\
  p_n &= e^{-k\lambda}\,\biggl[\,\sum_{j=1}^n \binom{n-1}{j-1}\,\frac{\lambda^j}{j!}
  \;\;-\;\; \sum_{i=1}^\ell (-1)^{i-1}\,\frac{\lambda^i}{i!} \sum_{j=0}^{n-i(k+1)} \binom{n - i(k+1) +i-1}{j+i-1}\,\frac{\lambda^j}{j!} \,\biggr]
  \\
  &\qquad\qquad (n = \ell(k+1)+m,\, m=0,1,\dots,k,\, \ell=1,2,\dots,\infty) \,.
\end{split}
\eq
See also an alternative combinatorial sum in \cite{Mane_Poisson_k_CC23_10}, for the case $n>k$.

\newpage
\setcounter{equation}{0}
\section{\label{sec:median1k}Equation for median}
As stated above, it was shown in \cite{Mane_Poisson_k_CC23_3} that for fixed $k\ge1$,
the median is zero if and only if $\lambda \le (\ln2)/k$, a result which is well-known for $k=1$.
We treat fixed $k\ge2$ below and vary the value of $\lambda$.
Denote the median simply by $\nu$ for brevity.
In this section, we consider the case where the median takes values in the interval $\nu\in[1,k]$.

\begin{proposition}
  For fixed $k\ge2$ and $\nu\in[0,k]$, the equation for $\lambda$ to determine the median is as follows
\bq  
\label{eq:eq_median_1k}
  \frac{e^{k\lambda}}{2} = \sum_{j=0}^\nu \binom{\nu}{j} \frac{\lambda^j}{j!} \,.
\eq
\end{proposition}
\begin{proof}
For $\nu\in[1,k]$, we employ eq.~\eqref{eq:KM_Thm1} to obtain the following equation for $\lambda$
\bq  
\begin{split}
  \frac{e^{k\lambda}}{2} &= 1 + \sum_{s=1}^\nu p_s
  \\
  &= 1 + \sum_{s=1}^\nu \sum_{j=1}^s \binom{s-1}{j-1} \frac{\lambda^j}{j!} 
  \\
  &= 1 + \sum_{j=1}^\nu \frac{\lambda^j}{j!} \biggl(\sum_{s=j}^\nu \binom{s-1}{j-1}\biggr)
  \\
  &= 1 + \sum_{j=1}^\nu \binom{\nu}{j} \frac{\lambda^j}{j!}
  \\
  &= \sum_{j=0}^\nu \binom{\nu}{j} \frac{\lambda^j}{j!} \,.
\end{split}
\eq
For $\nu=0$ the result is obvious, the equation to solve is simply $\frac12 e^{k\lambda}=1$.  
\end{proof}
\begin{remark}
  The equation for $\nu>k$ will not be discussed in this note.
  The expression for the sum $\sum_{s=1}^\nu p_s$ requires terms from the last line in eq.~\eqref{eq:KM_Thm1}, which contains complicated subtractions.
\end{remark}
\noindent
Even for $\nu=1$, eq.~\eqref{eq:eq_median_1k} has no simple solution.
The equation for $\nu=1$ is
\bq
\frac{e^{k\lambda}}{2} = 1 + \lambda \,.
\eq
The solution for $\lambda$ can be expressed in terms of the Lambert $W$ function,
but that cannot be expressed in terms of elementary functions.
We derive an iterative solution for $\lambda$ as follows.
Take the logarithm to obtain
\bq
k\lambda = \ln2 +\ln(1+\lambda) = \ln2 +\lambda -\frac{\lambda^2}{2} +\frac{\lambda^3}{3} +\dots
\eq
Rearrange to obtain
\bq  
\lambda = \frac{\ln2}{k-1} -\frac{\lambda^2}{2(k-1)} +\frac{\lambda^3}{3(k-1)} +\dots
\eq
Treat the terms on the right as small quantities and iterate to obtain a first approximation
\bq
\lambda \simeq \frac{\ln2}{k-1} \,.  
\eq
Iterate again to obtain
\bq
\lambda \simeq \frac{\ln2}{k-1} -\frac{(\ln2)^2}{2(k-1)^3} +\frac{(\ln2)^3}{3(k-1)^4} +\dots
\eq
And so on.
The higher order terms are more and more negligible as the value of $k$ increases.
In \cite{Mane_Poisson_k_CC23_3}, the solution for $\lambda$ for the case $\nu=0$ was denoted by $\lambda_* = (\ln2)/k$.
Extend that notation to $\lambda_{0,*}$ and denote the value of $\lambda$ in general by $\lambda_{\nu,*}$.
For $\nu=1$, the value of $\lambda_{1,*}$ for $k\gg1$ is approximately
\bq
\lambda_{1,*} \simeq \frac{\ln2}{k-1} = \frac{\ln2}{k}\biggl(1 +\frac{1}{k} +\frac{1}{k^2} +\dots\biggr) \,.
\eq
\begin{remark}
Proceeding in this way, we can obtain an iterative solution for $\lambda_{\nu,*}$ for $\nu=2$, etc.
For $k\gg1$ the solution is approximately
\bq
\lambda_{\nu,*} \simeq \frac{\ln2}{k-\nu} = \frac{\ln2}{k}\biggl(1 +\frac{\nu}{k} +\frac{\nu^2}{k^2} +\dots\biggr) \,.
\eq
For $\nu=0$, the above is an exact solution.
As is necessary, the value of $\lambda_{\nu,*}$ increases with $\nu$ (for fixed $k$).
However, this iterative solution only works for fixed $\nu$.
For $\nu=k$, the iterative solution fails because the right-hand side of eq.~\eqref{eq:eq_median_1k}
then does not contain a fixed number of summands as the value of $k$ changes.
\end{remark}

\newpage
\setcounter{equation}{0}
\section{\label{sec:median_eq_k}Median equals $k$}
Recall from Sec.~\ref{sec:median1k} that $\lambda_{\nu,*}$ is the value of $\lambda$ such that $P(Y_{k,\lambda} \le \nu) = \frac12$.
In this section, the notation ``$\mu$'' denotes the corresponding value of the mean $\mu_{\nu,*}=\kappa\lambda_{\nu,*}$.
The value of $\lambda_{\nu,*}$, thence $\mu_{\nu,*}$, was computed for each value of $\nu$ for fixed $k=20$, for values $\nu=[0,2k]$, i.e.~$\nu\in[0,40]$.
The resulting graph for $\mu$ is displayed in Fig.~\ref{fig:graph_k20_nu_mu}.
Observe that the graph changes shape at $\nu=k$:
it is a curve for $\nu\le k$ and (approximately) straight for $\nu\ge k$.
Similarly shaped graphs were obtained for all other tested values of $k$.
Hence the value $\nu=k$ represents a ``breakpoint'' where the shape of the graph of $\mu$ against $\nu$ changes structure.

In this section, we study the value of the mean $\mu$ for $\nu=k$, as a function of $k$.
Fig.~\ref{fig:graph_nu_eq_k_lambda} plots the of value of $\lambda$ (i.e.~$\lambda_{k,*}$) against $k$, for the case $\nu=k$, for $k\in[2,100]$.
The value of $\lambda$ decreases monitonically with $k$, approximately as $1/k$.
Fig.~\ref{fig:graph_nu_eq_k_mu} displays a plot of the mean $\mu$ against $k$, for the case $\nu=k$, for $k\in[2,10000]$.
The graph is visually a straight line, but it is not exactly so.
We know for the standard Poisson distribution, i.e.~$k=1$,
that given a value of the median $\nu$, the value of the mean $\mu$ such that $P(Y_{k,\lambda} \le \nu) = \frac12$ 
is close to but not exactly an affine function of $\nu$.
See \cite{AdellJodraPoisson2005,Choi,Alm2003,ChenRubin1986} and references therein.
See also the numerical results in \cite{Mane_Poisson_k_CC23_3}, for the Poisson distribution of order $k\ge2$, for the domain $\lambda\ge1$.

Following the analysis in \cite{Mane_Poisson_k_CC23_3}, let us define the difference between the mean and the median as
$\Delta_k = \mu-k$ (because $\nu=k$ in the present case).
Numerical fits to the data yield the following expression
\bq
\label{eq:Deltak_nu_eq_k}
\Delta_k \simeq 0.155752\,k + 0.57765625 - \frac{1}{16k} + \dots
\eq
Needless to say, the coefficients are approximate.
Nevertheless, eq.~\eqref{eq:Deltak_nu_eq_k} indicates that the deviation of $\mu-k$ from an affine function of $k$ is of $O(1/k)$ plus higher order terms.
This confirms the observation in Fig.~\ref{fig:graph_nu_eq_k_lambda} that the value of $\lambda_{k,*}$ decreases approximately as $1/k$.
The result is also consistent with the findings in \cite{Mane_Poisson_k_CC23_3}, which were for the domain $\lambda\ge1$.

Fig.~\ref{fig:graph_nu_eq_k_diff} displays a plot of the residual $\mu-k-\Delta_k$ for the data plotted in Fig.~\ref{fig:graph_nu_eq_k_mu}.
The residual would be zero if the expression for $\Delta_k$ in eq.~\eqref{eq:Deltak_nu_eq_k} were an exact fit to the data.
The value of the residual increases as the value of $k$ increases, which is to be expected from a numerical calculation.
The maximum difference is $\pm0.3$, at the right-hand edge $k=10000$, where $\mu\simeq11800$.
Observe also that the residual is {\em symmetric} around zero, i.e.~the fit in eq.~\eqref{eq:Deltak_nu_eq_k} is {\em unbiased}: it is neither systematically too high nor too low.

\newpage
\setcounter{equation}{0}
\section{\label{sec:median_scaling1}Scaling of median I}
In this section, we report numerical studies for the median in the interval $\nu\in[0,k]$.
Although an explicit equation to determine the value of $\lambda_{\nu,*}$ was given in eq.~\eqref{eq:eq_median_1k},
in general it is too difficult to solve analytically.
Specifically, we calculate (numerically) the value of the mean $\mu_{\nu,*}$ such that $P(Y_{k,\lambda} \le \nu) = \frac12$.

Recall that for $\nu=0$, then $\lambda_{0,*}=(\ln2)/k$, hence the $\mu_{0,*} = \kappa\lambda_{0,*} = \frac12(k+1)\ln2$.
Hence $\mu_{0,*}/(k+1) = \frac12\ln2$. This is an exact result for all $k\ge1$.
Hence we compute the value of $\mu_{\nu,*}/(k+1)$ as a function of $\nu/k$ for $\nu/k\in[0,1]$.
The result is displayed in Fig.~\ref{fig:graph_nu1k_mu}, for selected values of $k$.
Notice that the curves are similar in shape and lie close together.
This suggests that the expression for $\mu/(k+1)$ can be expressed as a series in powers of $\nu/k$:
\bq
\label{eq:nu1l_mu_series}
\frac{\mu_{\nu,*}}{k+1} = a_0 +a_1\frac{\nu}{k} +a_2\frac{\nu^2}{k^2} +a_3\frac{\nu^3}{k^3} +\dots
\eq
Here $a_0 = \frac12\ln2$ is exact and the other coefficients $a_i$ are expressed as series in powers of $1/(k+1)$.
Numerical studies yield the following expressions (rounding to three decimal places)
\bq
\label{eq:nu1l_mu_series_coeffs}
\begin{split}
  a_1 &\simeq 0.335 - \frac{0.014}{k+1} \,,
  \\
  a_2 &\simeq 0.356 + \frac{0.055}{k+1} \,,
  \\
  a_3 &\simeq 0.123 - \frac{0.7}{k+1} \,.
\end{split}
\eq
Fig.~\ref{fig:graph_nu1k_residual} displays a plot of the residual against $\nu/k$ for $\nu\in[0,k]$, for various values of $k$, 
where the residual is (using ``$\Delta$'' without subscript for lack of a better symbol)
\bq
\Delta = \frac{\mu_{\nu,*}}{k+1} -\biggl(a_0 +a_1\frac{\nu}{k} +a_2\frac{\nu^2}{k^2} +a_3\frac{\nu^3}{k^3} \biggr) \,.
\eq
Note the following:
\begin{enumerate}
\item
  The magnitude of $\Delta$ is small over the whole interval $\nu/k\in[0,1]$.
\item
  Other than the curve for $k=100$, the other curves are very close (for $k\ge 500$).
\item
  This suggests that the parameterization in eqs.~\eqref{eq:nu1l_mu_series} and \eqref{eq:nu1l_mu_series_coeffs}
  {\em converges uniformly} in the interval $\nu/k\in[0,1]$, as $k\to\infty$.
\item
  The curve for $k=100$ indicates that additional terms in powers of $1/(k+1)$ are required in the coefficients $a_i$ in eq.~\eqref{eq:nu1l_mu_series_coeffs}.
  Such additional terms are significant for $k=100$ but negligible for larger $k$, consistent with the observed behavior in Fig.~\ref{fig:graph_nu1k_residual}.
\end{enumerate}
\begin{remark}
  The fact that the interval $\nu/k\in[0,1]$ is a unit interval and the curves for $\mu/(k+1)$ are smooth suggests that a parameterization
  using a family of orthogonal polynomials might be a good idea.
  Note however that the value of $\nu$ is an integer, hence $\nu/k$ is actually a discrete valued variable,
  hence cannot necessarily be evaluated exactly at the roots of a Legendre or Chebyshev polynomial, for example.
\end{remark}

\newpage
\setcounter{equation}{0}
\section{\label{sec:median_scaling2}Scaling of median II}
In this section, we study the median for values $\nu\ge k$.
Previously in this note, we fixed the value of the median $\nu$ and
computed the corresponding values for $\lambda$ and the mean $\mu$ such that $P(Y_{k,\lambda} \le \nu) = \frac12$.
Here we shift the paradigm slightly.
We follow the analysis in \cite{Mane_Poisson_k_CC23_3}.
We fix the value of the mean to be an integer, say $\mu=n$ (hence $\lambda=n/\kappa$), and compute the corresponding value of the median $\nu$.
For values $\lambda\ge1$, i.e.~$n\ge\kappa$, a conjectured expression for the median was derived (via numerical studies) in \cite{Mane_Poisson_k_CC23_3}.
See eq.~\eqref{eq:Poisson_k_median_int_n}.
In this section, we treat the domain $\lambda < 1$, i.e.~$n\le\kappa$.
It was stated in \cite{Mane_Poisson_k_CC23_3} that eq.~\eqref{eq:Poisson_k_median_int_n} is not always valid for $\lambda<1$.
Given an integer $n$, let us {\em define} the ``base median'' as the expression in eq.~\eqref{eq:Poisson_k_median_int_n}:
\bq
\label{eq:base_median_def}
\nu_{\rm base} := n - \biggl\lfloor\frac{k+4}{8}\biggr\rfloor \,.
\eq
We compute the actual value of the median $\nu$ and plot the value of the difference $(\nu_{\rm base}-\nu)/k$ against $n/k$, where recall the value of the mean is $\mu=n$.
Fig.~\ref{fig:graph_k100-500} displays the result for $k=100,200,300,400,500$.
Since the value of $\nu_{\rm base}-\nu$ is always an integer, the graphs are clearly step functions for small values of $k$, but smooth out as $k$ increases.
Nevertheless, the shapes of the curves in Fig.~\ref{fig:graph_k100-500} suggest the curves approach the same limiting shape for large $k$.
This is confirmed in Fig.~\ref{fig:graph_k1000-1000_full}, where the value of $(\nu_{\rm base}-\nu)/k$ is plotted against $n/k$ for $k=1000,2000,5000,10000$.
The curves are much closer together and appear to converge to a common curve as $k\to\infty$.
A zoom view of Fig.~\ref{fig:graph_k1000-1000_full} is displayed in Fig.~\ref{fig:graph_k1000-1000_zoom}, to display the curves in more detail, for $n/k \le 5$.

\begin{remark}
  The left edge of the curves in
  Figs.~\ref{fig:graph_k100-500}, \ref{fig:graph_k1000-1000_full} and \ref{fig:graph_k1000-1000_zoom} is set to the value of $n$ at which the median equals zero.
  At this point, $\mu = \frac12(k+1)\ln2$, hence
\bq
  \frac{\nu_{\rm base}}{k} \simeq \frac{\ln2}{2}\frac{k+1}{k} - \frac{k+4}{8k} \simeq \frac{\ln2}{2} - \frac18 \simeq 0.221574 \,.
\eq
This explains why all the curves in Figs.~\ref{fig:graph_k100-500}, \ref{fig:graph_k1000-1000_full} and \ref{fig:graph_k1000-1000_zoom} begin with a value
of approximately $0.22$.
This also indicates that the domain of values $n/k$ actually extends into the region $n<k$,
as evidenced by the values on the horizontal axis in Figs.~\ref{fig:graph_k100-500}, \ref{fig:graph_k1000-1000_full} and \ref{fig:graph_k1000-1000_zoom}.
This goes to show that the value of the median exhibits many scaling properties.
In all cases, the relevant quantities to calculate are $\nu/k$ and $\mu/k$.
\end{remark}

\newpage
\section{\label{sec:conc}Conclusion}
This note analyzed some properties of the median $\nu$ of the Poisson distribution of order $k\ge2$.
Given a value for the median in the interval $\nu\in[1,k]$, an equation to calculate the corresponding value of the rate parameter $\lambda$ was derived.
Numerical evidence was also presented that the value of the median exhibits many scaling properties,
which permit one to formulate parameterizations of the value of the median in various domains of the parameter space $(k,\lambda)$.
In all cases, the relevant quantities to calculate are $\nu/k$ and $\mu/k$, where $\mu$ is the mean.


\newpage
\begin{figure}[!htb]
\centering
\includegraphics[width=0.75\textwidth]{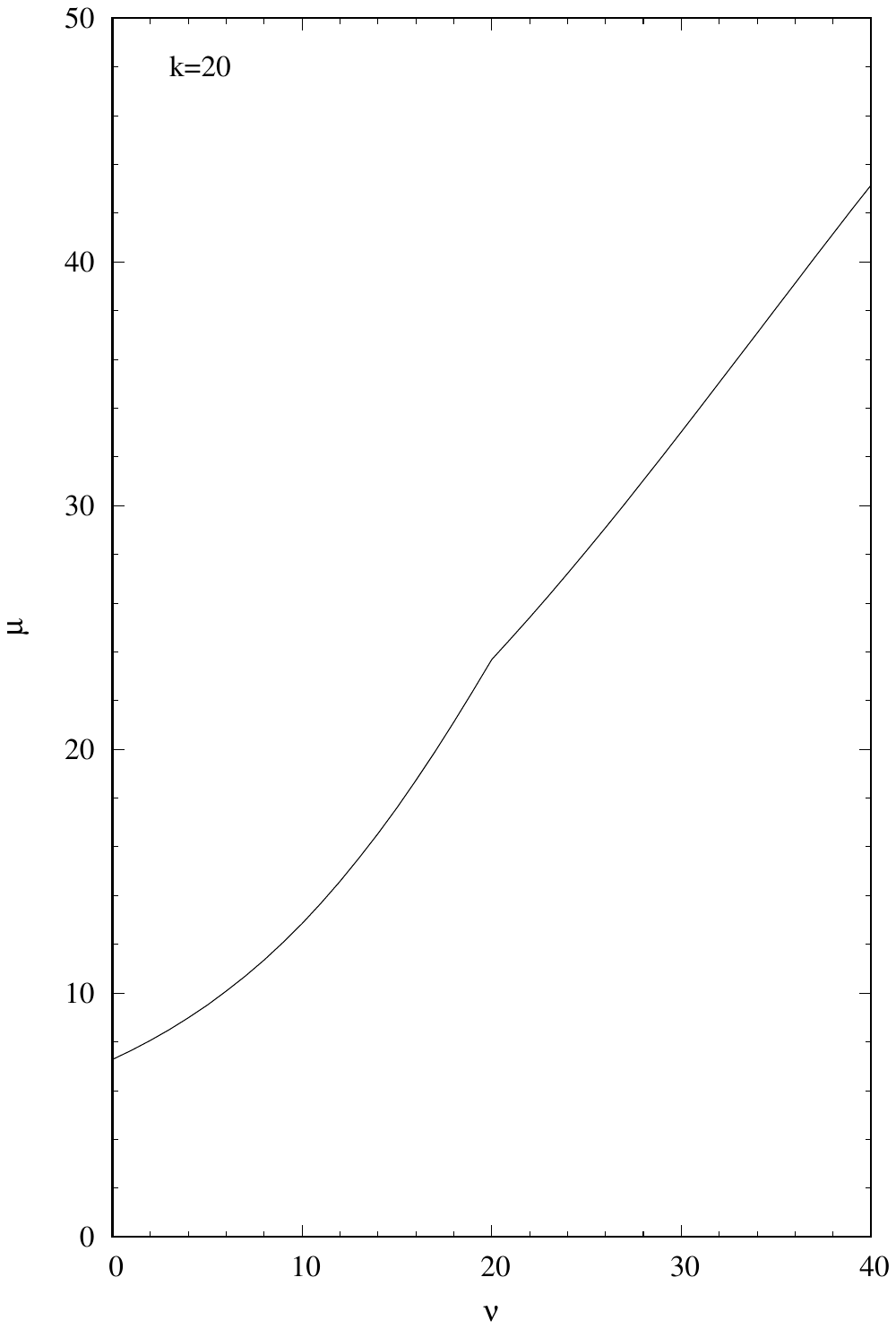}
\caption{\small
\label{fig:graph_k20_nu_mu}
Plot of the mean $\mu$ against the median $\nu$ for $\nu=0,\dots,40$, for fixed $k=20$.}
\end{figure}

\newpage
\begin{figure}[!htb]
\centering
\includegraphics[width=0.75\textwidth]{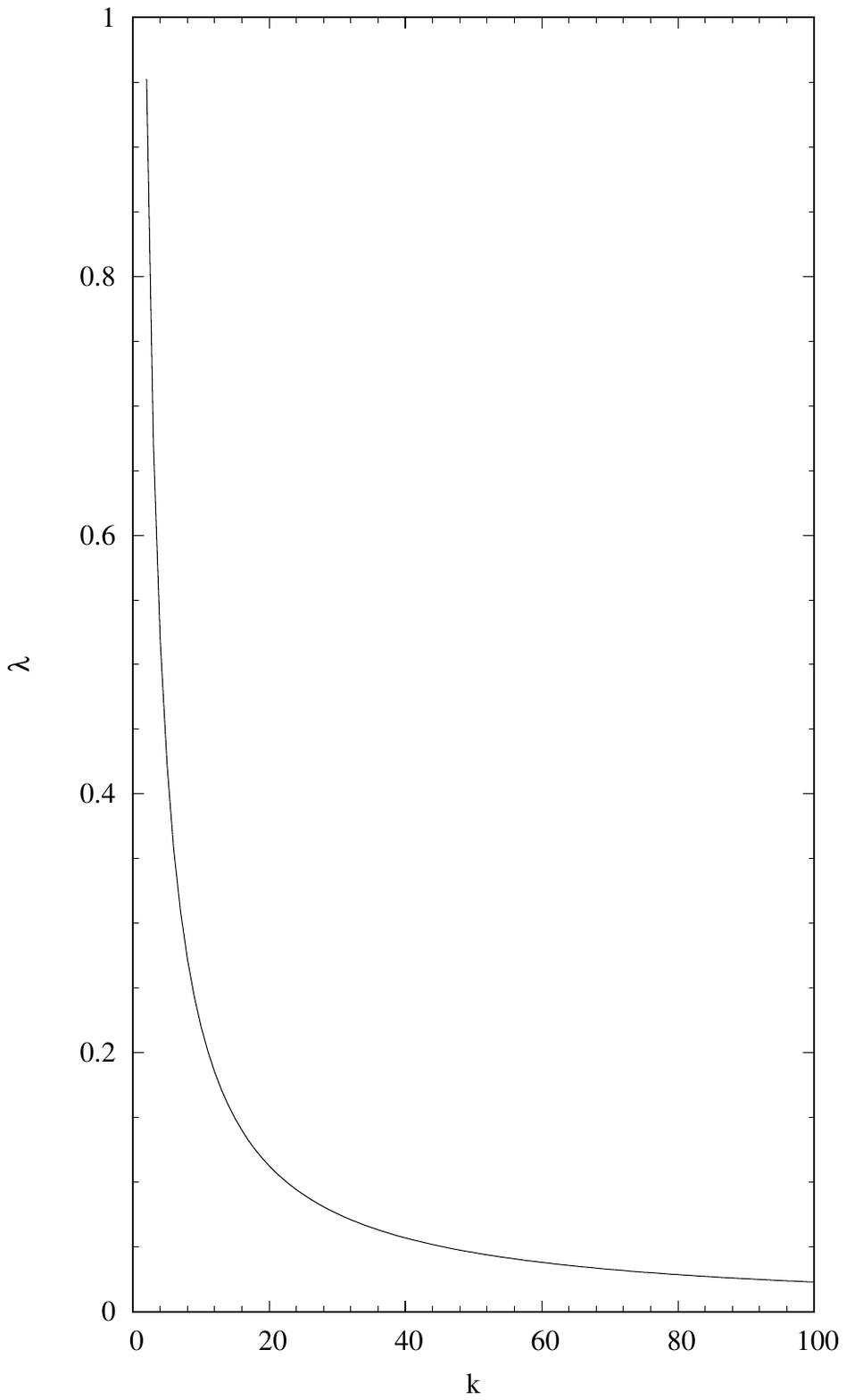}
\caption{\small
\label{fig:graph_nu_eq_k_lambda}
Plot of value of $\lambda$ against $k$, for the case $\nu=k$, for $k\in[2,100]$.}
\end{figure}

\newpage
\begin{figure}[!htb]
\centering
\includegraphics[width=0.75\textwidth]{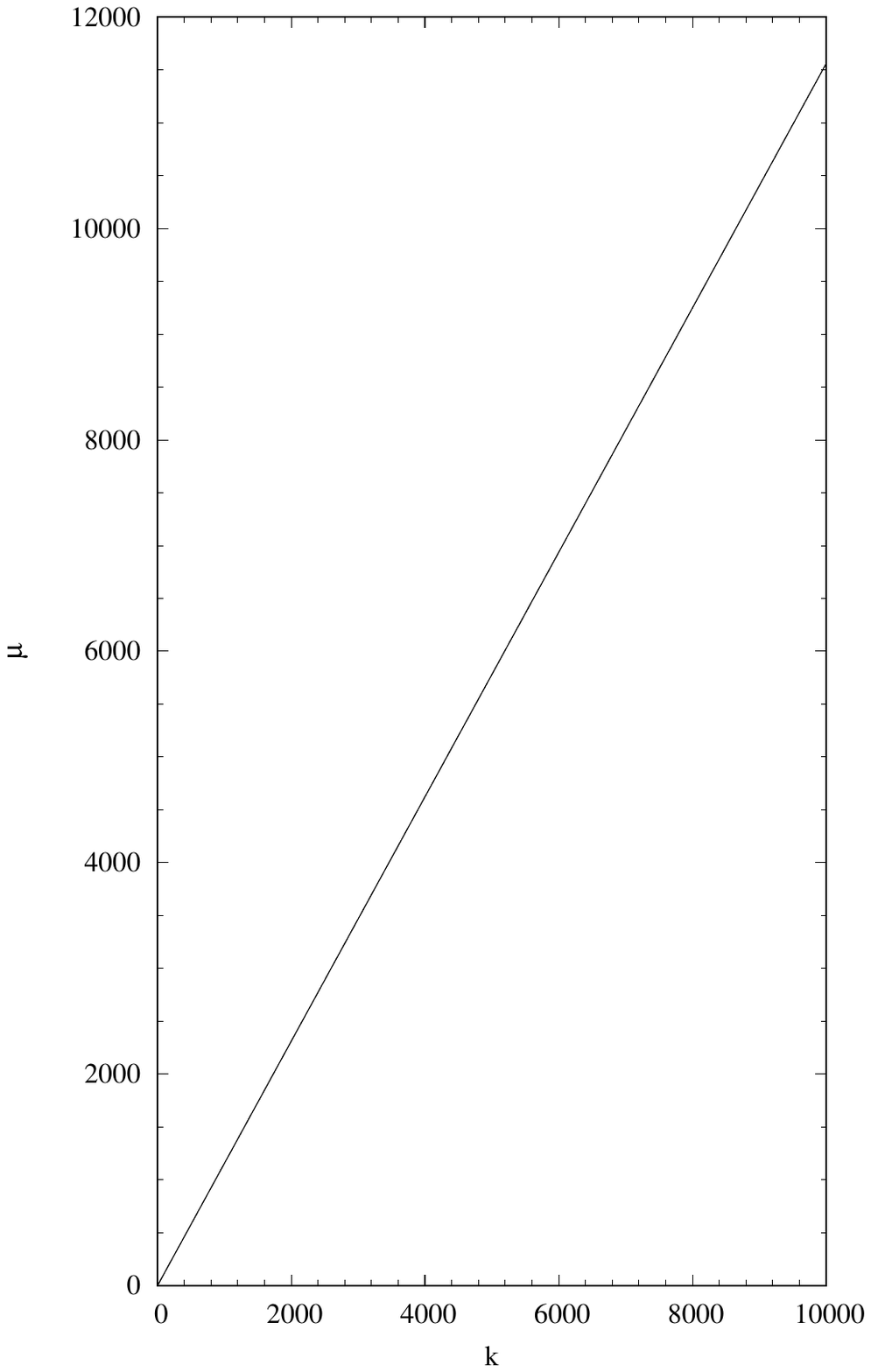}
\caption{\small
\label{fig:graph_nu_eq_k_mu}
Plot of the mean $\mu$ against $k$, for the case $\nu=k$, for $k\in[2,10000]$.}
\end{figure}

\newpage
\begin{figure}[!htb]
\centering
\includegraphics[width=0.75\textwidth]{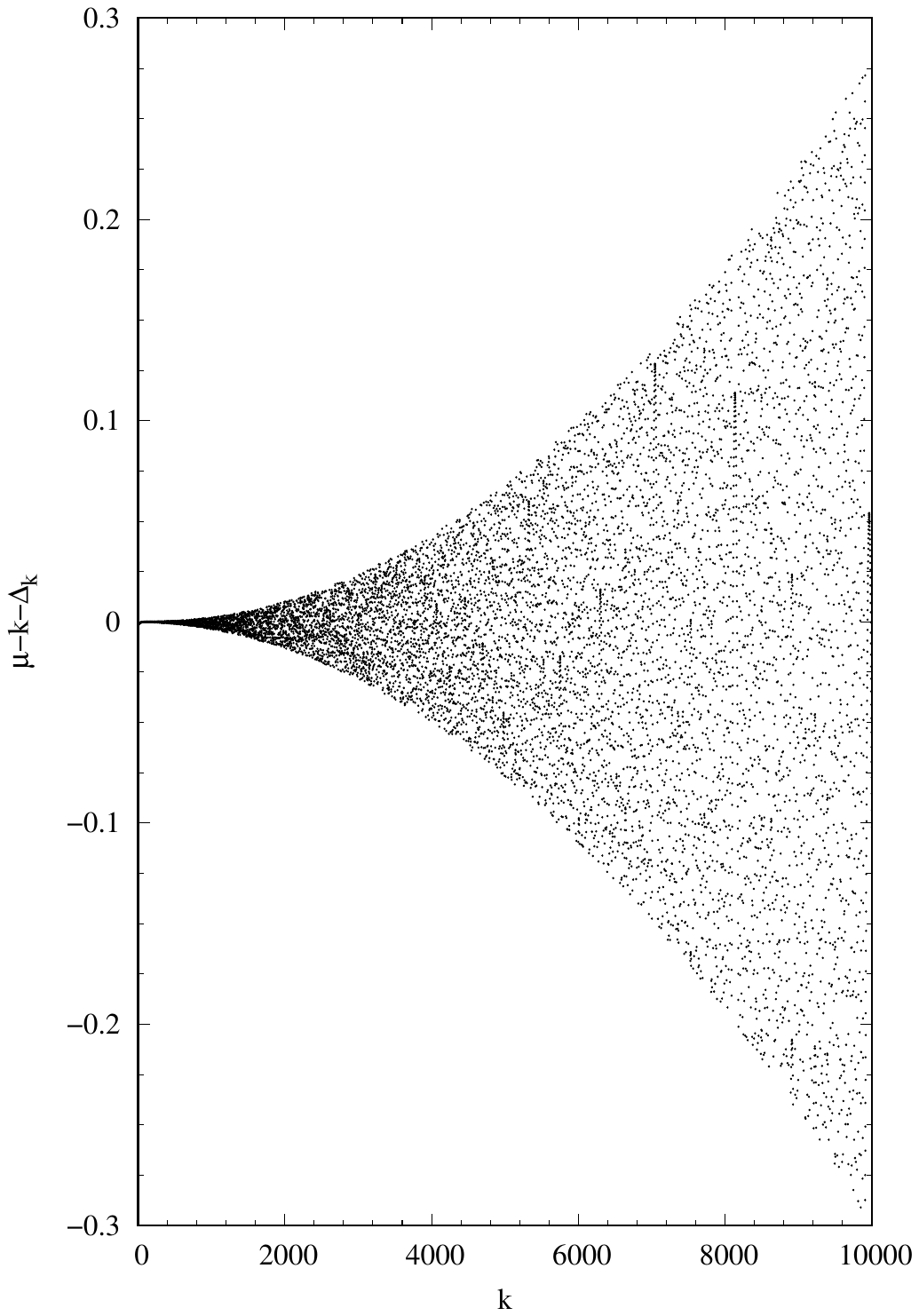}
\caption{\small
\label{fig:graph_nu_eq_k_diff}
Plot of the residual $\mu-k-\Delta_k$ against $k$, for the case $\nu=k$, for $k\in[2,10000]$.}
\end{figure}

\newpage
\begin{figure}[!htb]
\centering
\includegraphics[width=0.75\textwidth]{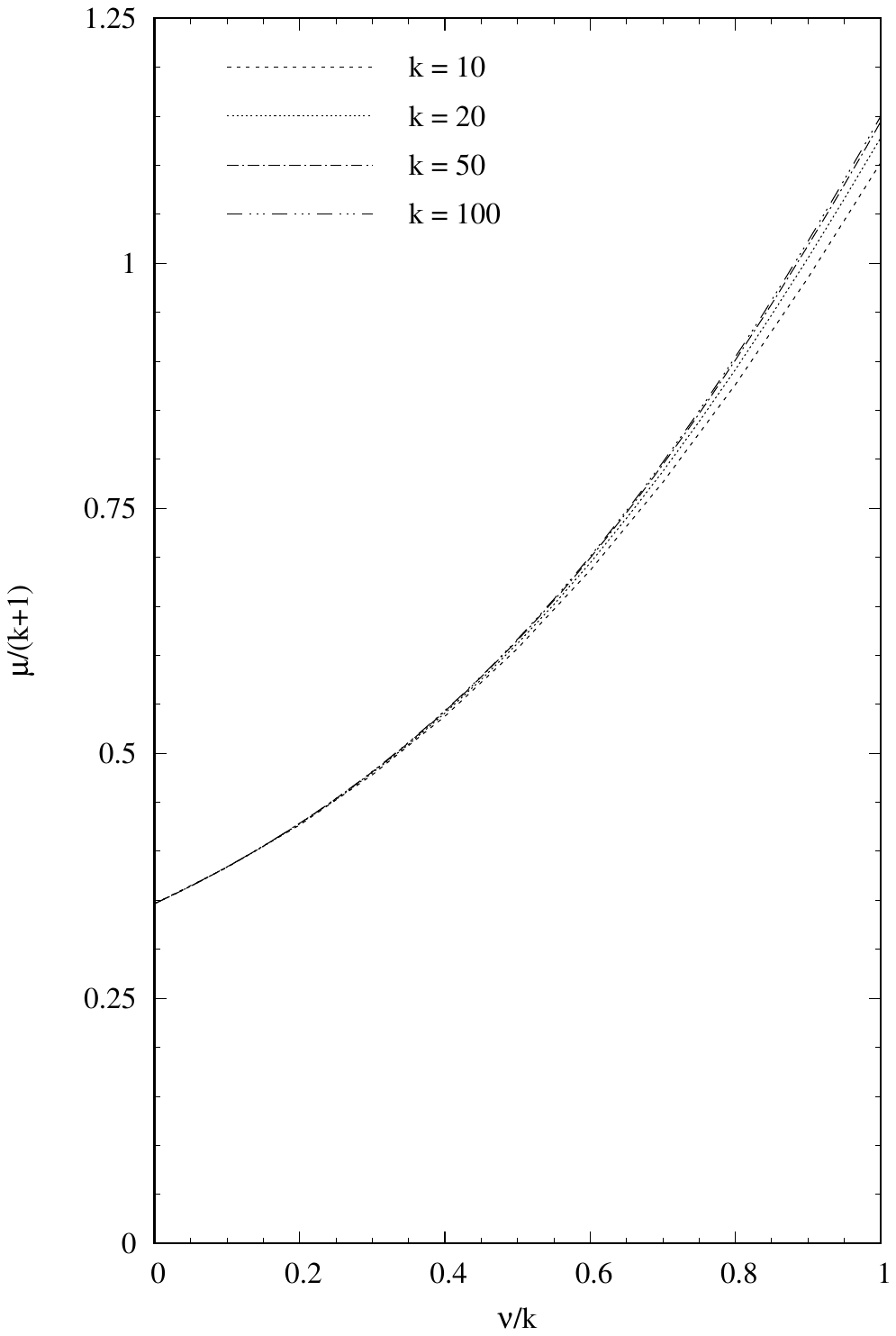}
\caption{\small
\label{fig:graph_nu1k_mu}
Plot of $\mu/(k+1)$ against $\nu/k$, for $\nu/k\in[0,1]$, for selected values of $k$.}
\end{figure}

\newpage
\begin{figure}[!htb]
\centering
\includegraphics[width=0.75\textwidth]{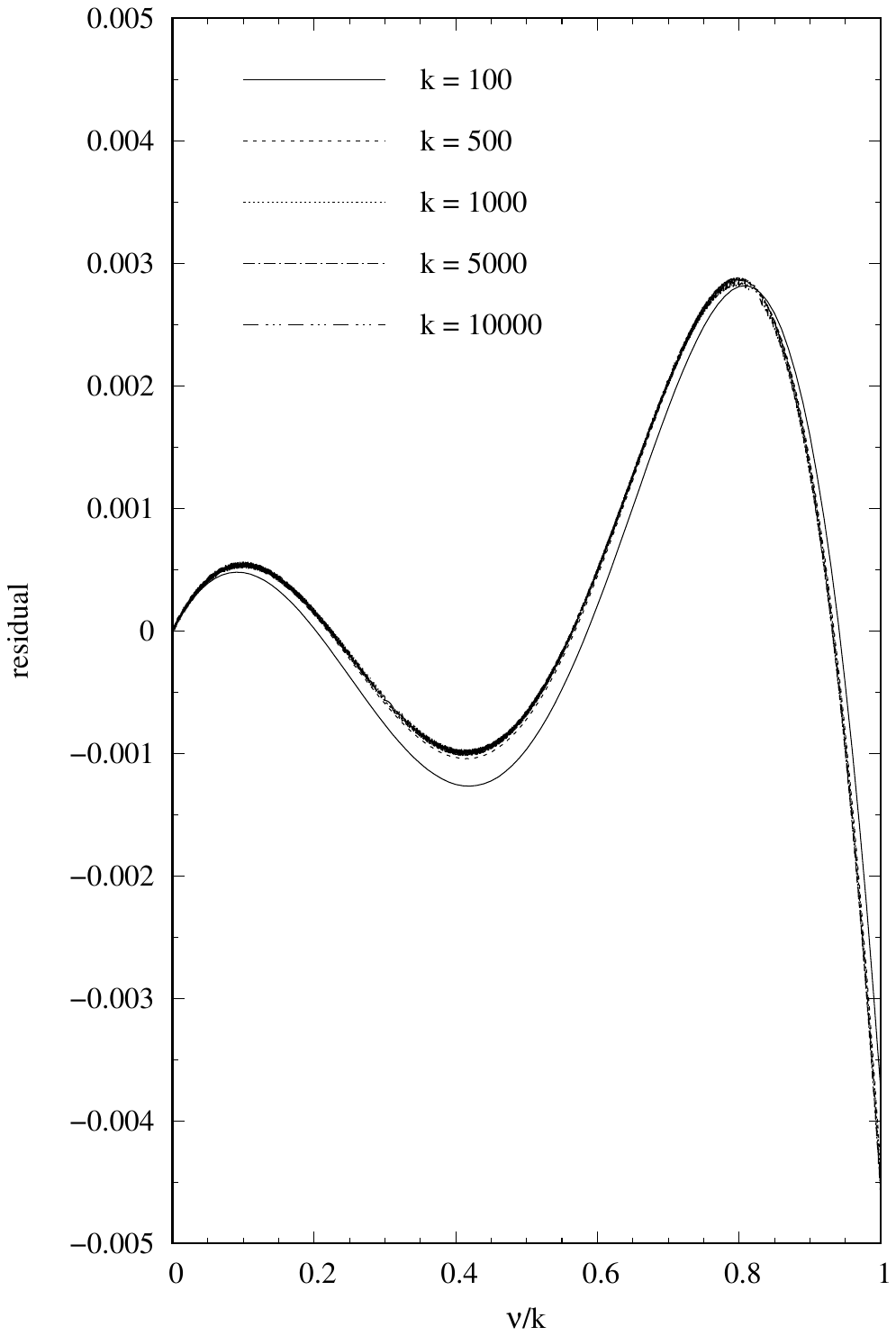}
\caption{\small
\label{fig:graph_nu1k_residual}
Plot of the residual fit for the scaled mean $\mu/(k+1)$ against $\nu/k$, for $\nu/k\in[0,1]$, for selected values of $k$.}
\end{figure}

\newpage
\begin{figure}[!htb]
\centering
\includegraphics[width=0.75\textwidth]{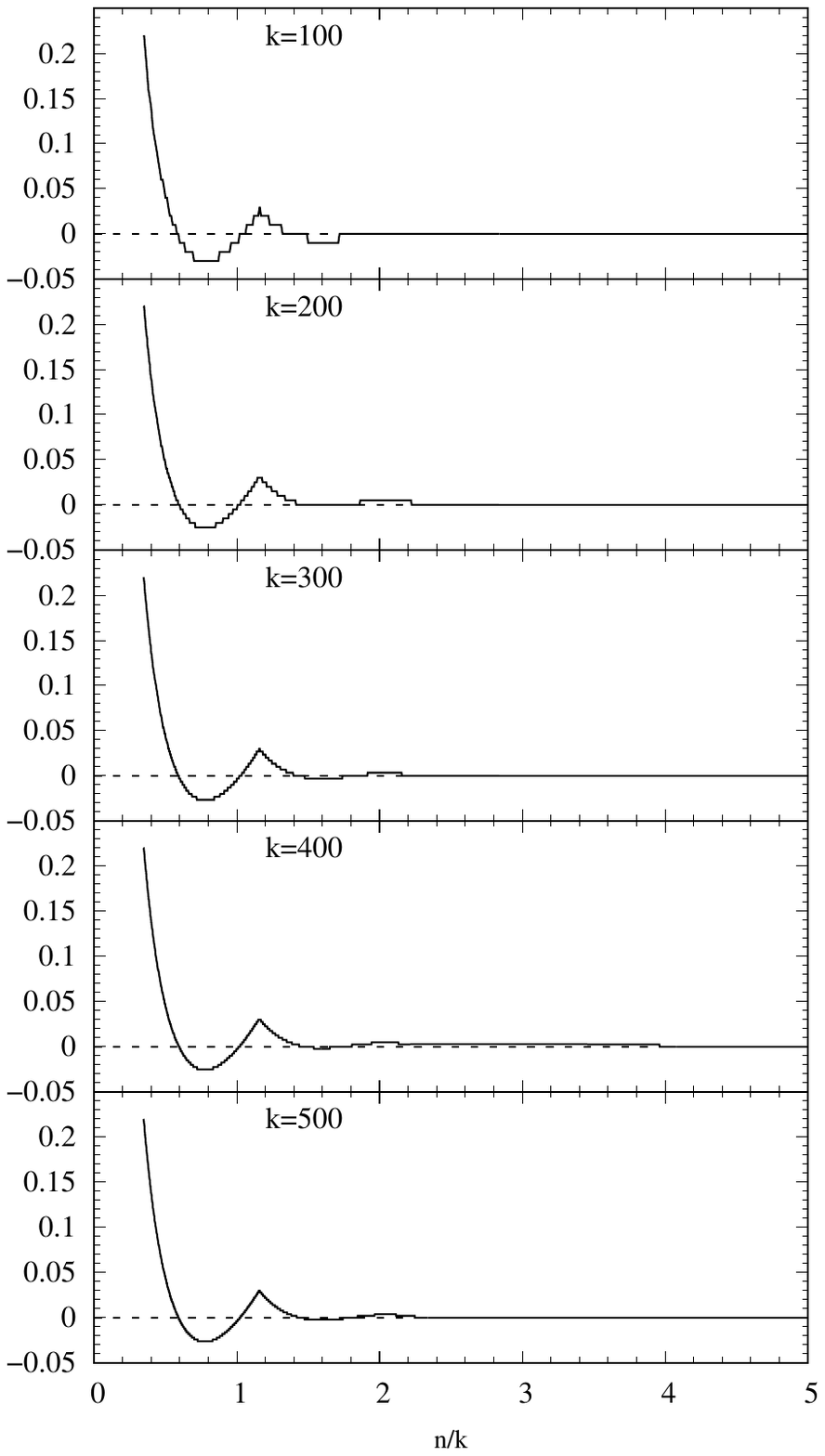}
\caption{\small
\label{fig:graph_k100-500}
Plot of the difference $(\nu_{\rm base}-\nu)/k$ against $n/k$ (where the mean equals $n$) for selected values of $k$.}
\end{figure}

\newpage
\begin{figure}[!htb]
\centering
\includegraphics[width=0.75\textwidth]{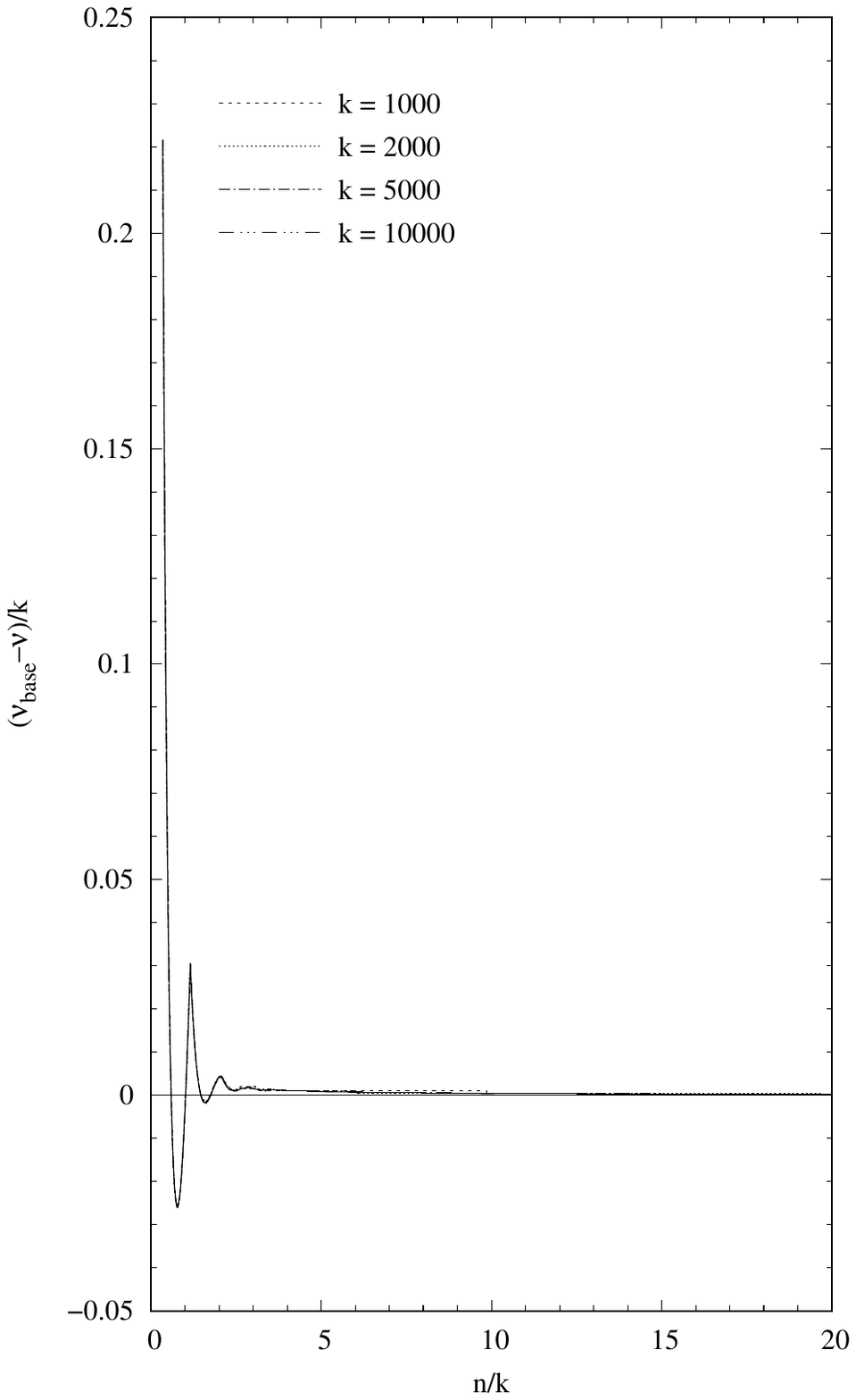}
\caption{\small
\label{fig:graph_k1000-1000_full}
Plot of the difference $(\nu_{\rm base}-\nu)/k$ against $n/k$ (where the mean equals $n$) for selected values of $k$.}
\end{figure}

\newpage
\begin{figure}[!htb]
\centering
\includegraphics[width=0.75\textwidth]{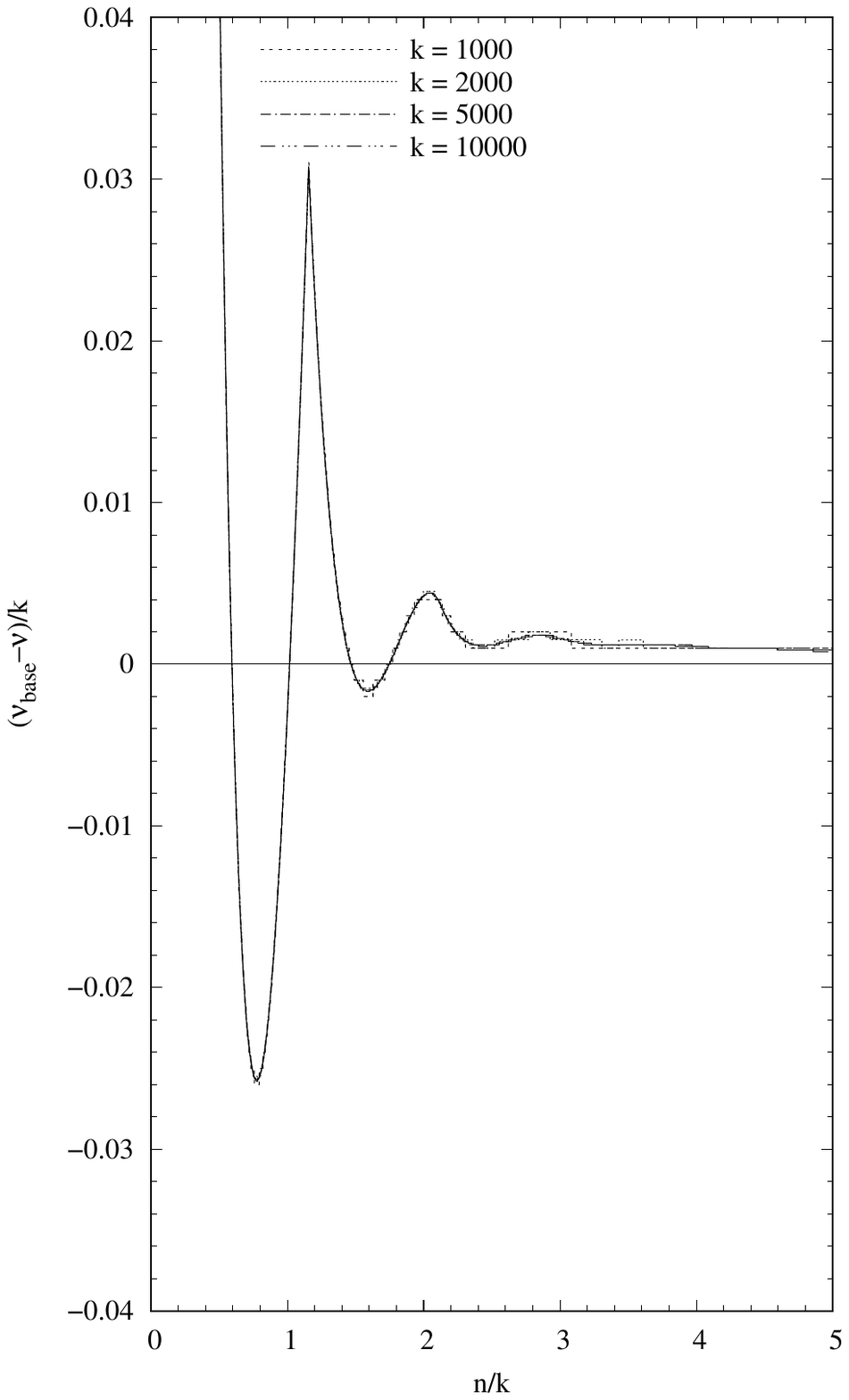}
\caption{\small
\label{fig:graph_k1000-1000_zoom}
Plot of the difference $(\nu_{\rm base}-\nu)/k$ against $n/k$ (where the mean equals $n$) for selected values of $k$.
A zoom view of Fig.~\ref{fig:graph_k1000-1000_full}.}
\end{figure}

\end{document}